\documentclass[a4paper,12pt,showkeys]{amsart}
\usepackage{amsmath,amssymb,amsthm,mathrsfs}
\usepackage{type1cm}
\usepackage{ulem}
\usepackage{fancybox}
\usepackage[dvipdfmx]{graphicx}
\usepackage{tikz}
\usepackage[all]{xy}
\usepackage{array,booktabs}
\usepackage{bm}
\usepackage{wrapfig}

\theoremstyle{plain}
\newtheorem{theorem}{Theorem}[section]
\newtheorem{corollary}[theorem]{Corollary}
\newtheorem{lemma}[theorem]{Lemma}
\newtheorem{conjecture}[theorem]{Conjecture}

\theoremstyle{definition}
\newtheorem{definition}[theorem]{Definition}
\newtheorem{remark}[theorem]{Remark}
\newtheorem{example}[theorem]{Example}

\newcommand{\AS}{{\rm{AS}}}
\newcommand{\SAS}{{\rm{SAS}}}
\newcommand{\alt}{{\rm{alt}}}

\begin{document}

\title{On knot semigroups and Gelfand-Kirillov dimensions.}
\author{Toshinori Miyatani}
\date{}
\address{Graduate School of Science, Hokkaido University, Sapporo, 004-0022, Japan}
\email{miyatani@math.sci.hokudai.ac.jp}

\begin{abstract}
A knot semigroup is defined by A. Vernitski. A. Vernitski conjectured that the knot semigroup of the 2-bridge knot is isomorphic to an alternating sum semigroup. To support this conjecture, and as a first main result, we prove that the knot semigroup of the double twist knot is isomorphic to an alternating sum semigroup. Moreover, as a second main result, we construct a link invariant constructed by the Gelfand-Kiriliov dimension of algebra.
\end{abstract}


\maketitle{}
\tableofcontents

\section{Introduction}
We shall consider the theory of knot semigroups. A knot semigroup is a cancellative semigroup whose defining relations come in pairs of the form $xy=yx$ and $yx=zy$ from crossing points of a knot diagram. This construction is similar to the Wirtinger presentation of the knot group. The knot semigroups are defined by A. Vernitski. A. Vernitski proved that torus knots and twist knots are isomorphic to alternating sum semigroups and conjectured that 2-bridge knot is isomorphic to an alternating sum semigroup \cite{V1}. As a main result, and to support this conjecture, we shall prove that the double twist knot is isomorphic to an alternating sum semigroup. Next, we consider the growth of knot semigroups. To investigate the growth of knot semigroups, we use a growth function of semigroups and the Gelfand-Kirillov dimension of semigroup algebra. As a second main result, we construct a link invariant constructed by the Gelfand-Kirillov dimension of algebra. This research is a first step connecting knot theory and semigroup theory.

This paper is organized as follows. In section 2, we recall standard definitions. In section 2.1, we recall a knot semigroup defined by A. Vernitski. In section 2.2, we review an alternating sum semigroup. In section 2.3, we recall a definition of the 2-bridge knot and its properties. In section 3, we describe examples of knot semigroups proved by A. Vernitski. In section, 3.1 we see that the knot semigroup of a trivial knot is the semigroup of the set of natural numbers. In section 3.2, we recall that the   knot semigroup of the torus knot is an alternating sum semigroup. In section 3.3, we review that the knot semigroup of the twist knot is isomorphic to an alternating sum semigroup. In section 4.2, we state the first main result. In section 4.3, we shall prove that the knot semigroup of the double twist knot is isomorphic to an alternating sum semigroup. In section 5 we discuss a growth function of knot semigroups. In section 6.1 we review the Gelfand-Kirillov dimensions. In section 6.3 we compute examples.
\section{Preliminaries}
\subsection{Knot semigroups}
We recall a knot semigroup defined by A. Vernitski \cite{V1}. First we define the cancellative semigroups.
\begin{definition}
A semigroup $S$ is called cancellative if it satisfies two conditions : if $xz=yz$ then $x=y$, and if $xy=xz$ then $y=z$ for $x, y, z \in S$.
\end{definition}
\begin{wrapfigure}[8]{r}[20pt]{5cm}
\includegraphics[keepaspectratio,width=5cm,bb=0 0 250 120]{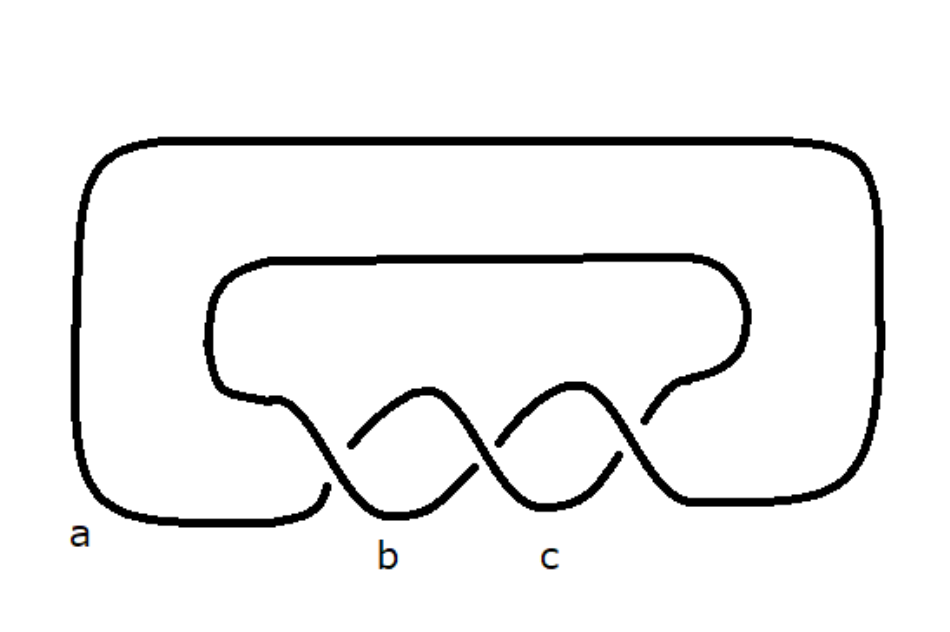}
\begin{center}Figure 1 : $T(2,3)$\end{center}
\end{wrapfigure}
By an arc we mean a continuous line on a knot diagram from one undercrossing to another undercrossing.
For example, consider the knot diagram $T(2,3)$ on figure 1.
It has three arcs, denoted by $a, b$, and $c$. 

Let $K$ be a knot diagram. We shall define a semigroup, which we call the knot semigroup of $K$, and denote by $M(K)$. We assume that each arc is denoted by a letter. Then we define two defining relations $xy=yz$ and $yx=zy$ at crossing, where arcs $x$ and $z$ form the undercrossing and arc $y$ is the overcrossing. We define these relations at every crossing. The cancellative semigroup generated by arc letters with these defining relations is the knot semigroup of the knot diagram. This construction is the analogy of the Wirtinger presentation of knot group. \cite{W}.

\begin{wrapfigure}[7]{r}[20pt]{5cm}
\includegraphics[keepaspectratio,width=5cm,bb=0 0 250 120]{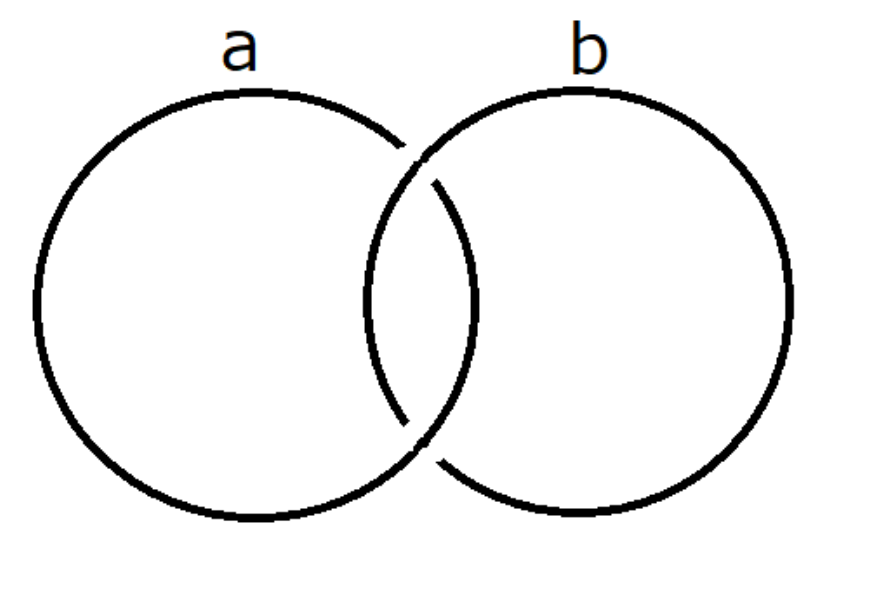} 
\begin{center}Figure 2 : Hopf links\end{center}
\end{wrapfigure}
The definition of the knot semigroup naturally generalizes from diagrams of knots to diagrams of links. For example, the diagram of the Hopf link in Figure 2 contains two arcs $a, b$ and two crossings, each defining a single relation $ab=ba$. Hence, its knot semigroup is a free commutative semigroup with two generators. 

\subsection{Alternating sum semigroups}
We shall recall an alternating sum semigroup defined by A. Vernitski \cite{V1}. Let $G$ be either $\mathbb{Z}_n$ or $\mathbb{Z}$. Let $B$ be a subset of $G$, and $B^+$ the set of words of B. By the alternating sum of a word $b_1b_2b_3b_4\dots b_k \in B^+$ we shall define the value of $b_1-b_2+b_3-b_4+\dots+(-1)^{k+1}b_k$ calculated in $G$. The notation ${\rm{alt}}(b_1b_2\dots b_k)$ denote the value, i.e.,
\begin{equation*}
{\rm{alt}}(b_1b_2\dots b_k)=b_1-b_2+b_3-b_4+\dots+(-1)^{k+1}b_k.
\end{equation*}
We also define the following notation. $|b_1b_2\dots b_k|$ denote the length of the word $b_1b_2\dots b_k$. We shall say that two words $u,v\in B^+$ are in relation $\sim$ if and only if
\begin{enumerate}
\item[(1)]$|u|=|v|$,
\item[(2)]${\rm{alt}}(u)={\rm{alt}}(u)$.
\end{enumerate}
The relation $\sim$ is a congruence on $B^+$. The semigroup $\AS(G, B)$ denote the factor semigroup $B^+/\sim$. $\AS(G, B)$ is called an alternating sum semigroup. 

We shall also recall a strong alternating sum semigroup defined by A. Vernitski \cite{V1}. The sets $G$ and $B$ are as above. Let us say that $g\in G$ is even (resp. odd) in $G$ if $g$ can be represented in the form $g=2h$ (resp. $g=2h+1$) for some $h\in G$. Let $w\in B^+$. The notation $|w|_e$ denote the number of entries in $w$ which are even in $G$. We shall say that two words $u, v \in B^+$ are in relation $\approx$ if and only if 
\begin{enumerate}
\item[(1)]$|u|=|v|$,
\item[(2)]$\alt(u)=\alt(v)$,
\item[(3)]$|u|_e=|v|_e$.
\end{enumerate}  
The relation $\approx$ is a congruence on $B^+$. The semigroup $\SAS(G, B)$ denote the factor semigroup $B^+/\approx$. $\SAS(G, B)$ is called a strong alternating sum semigroup. 
\subsection{2-bridge knots}
\subsubsection{Definitions of 2-bridge knots}
We recall the 2-bridge knot. We first define a bridge number of a knot diagram.
\begin{definition}
Suppose that $D$ is a knot diagram of a knot (or link) $K$. If we can divide up $D$ into $2n$ polygonal curves $\alpha_1,\alpha_2,\dots,\alpha_n$ and $\beta_1,\beta_2,\dots,\beta_n$, i.e.,  
\begin{equation*}
D=\alpha_1\cup\alpha_2\cup\dots\cup\alpha_n\cup\beta_1\cup\beta_2\cup\dots\cup\beta_n
\end{equation*}
that satisfy the conditions given by the followings, then the bridge number of $D$, ${\rm{br}}(D)$, is said to be at most $n$ and denoted by ${\rm{br}}(D)\le n$. 
\begin{enumerate}
\item[(1)]$\alpha_1,\alpha_2,\dots,\alpha_n$ are mutually disjoint, simple polygonal curves.
\item[(2)]$\beta_1,\beta_2,\dots,\beta_n$ are also mutually disjoint, simple curves.
\item[(3)]At the crossing points of $D$, $\alpha_1,\alpha_2,\dots,\alpha_n$ are segments that passes over the crossing points. While at the crossing points of $D$, $\beta_1,\beta_2,\dots,\beta_n$ are segments that pass under the crossing points.
\end{enumerate}
If ${\rm{br}}\le n$ but ${\rm{br}}\nleq n-1$, then we define ${\rm{br}}(D)=n$.
\end{definition}
The bridge number of a knot diagram $D$ is not a knot invariant for a knot $K$. But we have the following theorem.
\begin{theorem}
Let $K$ be a knot or a link. Then the number ${\rm{br}}(K)=\displaystyle \min_\mathcal{D}$ is an invariant for $K$, where $\mathcal{D}$ is the set of all knot diagrams of $K$. The number ${\rm{br}}(K)$ is called the bridge number of $K$.
\end{theorem}
Let $K$ be a knot or a link. If ${\rm{br}}(K)=2$, then $K$ is called a 2-bridge knot.
\subsubsection{Conway's normal forms}
Any 2-bridge knot has a presentation, which can be deformed as in figure, where $a_i$ indicates $|a_i|$ $(\neq 0)$ crossing points with sign $\epsilon_i=a_i/|a_i|=\pm 1$. 
\begin{center}
\includegraphics[keepaspectratio,width=5cm,bb=170 0 420 130]{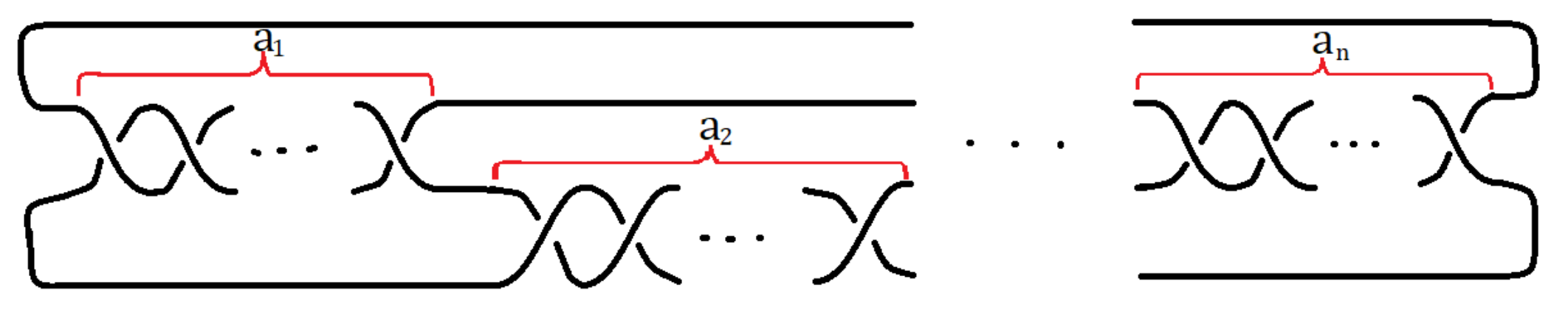}
\end{center}
\begin{center}
($n$ is even) 
\end{center}
\begin{center}
\includegraphics[keepaspectratio,width=5cm,bb=170 0 420 130]{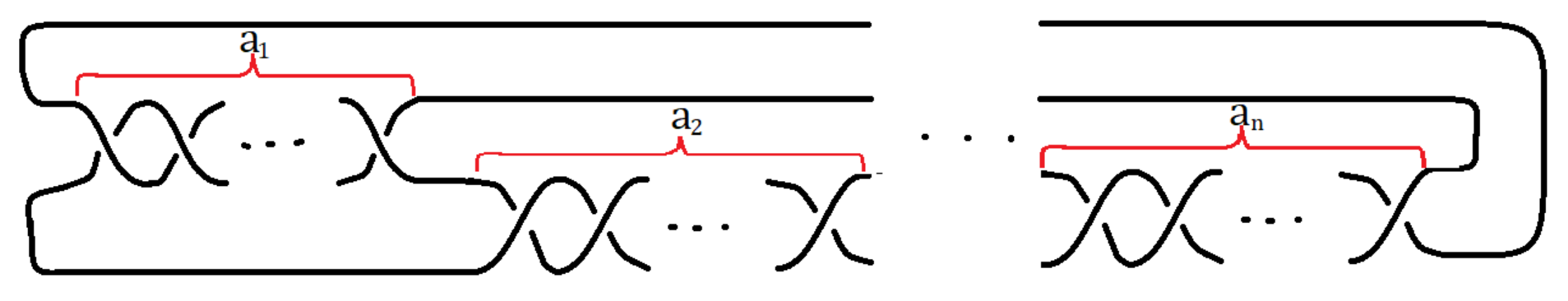}
\end{center}
\begin{center}
($n$ is odd) 
\end{center}
We denote the 2-bridge knot with this knot diagram by $C(a_1,a_2,\dots,a_n)$, which is called Conway's normal form.

\section{Examples of knot semigroups}
\subsection{Trivial knots}
\begin{wrapfigure}[6]{r}[20pt]{5cm}
\includegraphics[keepaspectratio,width=4cm,bb=0 0 250 150]{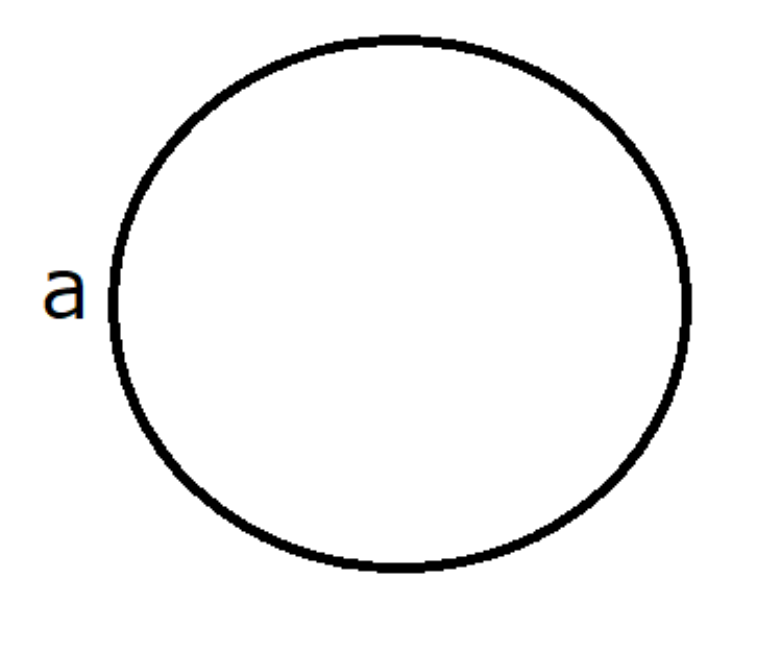}
\begin{center}Figure3 : Trivial knots\end{center}
\end{wrapfigure}
Let $\mathbb{N}$ be the semigroup of positive integers. The semigroup $\mathbb{N}$ is a cancellative semigroup. The diagram of the trivial knot contains one arc and no crossings (Figure 3). Therefore, its knot semigroup is isomorphic to the semigroup $\mathbb{N}$. 
\subsection{Torus knots and torus links}
\begin{wrapfigure}[10]{r}[20pt]{6cm}
\includegraphics[keepaspectratio,width=4cm,bb=0 0 250 220]{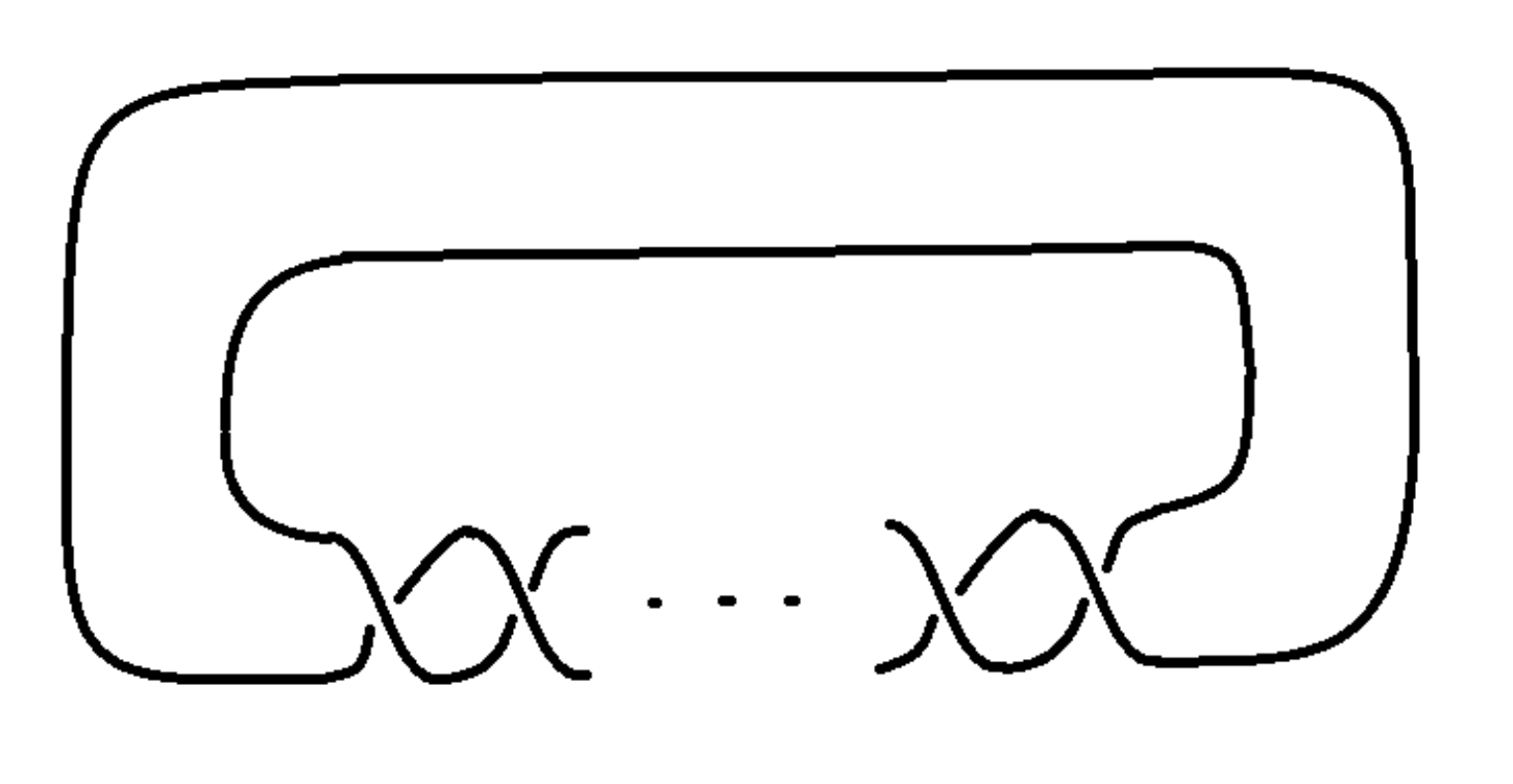}
\begin{center}Figure4 : Torus knots\end{center}
\end{wrapfigure}
A torus knot $T(2, n)$ consists of $n$ half-twists (Figure 4.). We recall the knot semigroup $M(T(2, n))$ of a knot diagram $T(2, n)$ (with an odd $n$) and the knot semigroup of a link diagram $T(2, n)$ (with an even $n$) proved by A. Vernitski \cite{V1}. 
\begin{theorem}[\cite{V1} Theorem 3.] \label{torusthm}
Let n be an odd integer. The knot semigroup $M(T(2, n))$ of the torus knot diagram $T(2,n)$ is isomorphic to the alternating sum semigroup $\AS(\mathbb{Z}_n,\mathbb{Z}_n)$.
\end{theorem}
\begin{theorem}[\cite{V1} Theorem 13.]
Let n be an even integer. The knot semigroup $M(T(2, n))$ of the torus link diagram $T(2, n)$ is isomorphic to the strong alternating sum semigroup $\SAS(\mathbb{Z}_n,\mathbb{Z}_n)$.
\end{theorem}
Since $\SAS(\mathbb{Z}_n,\mathbb{Z}_n)=\AS(\mathbb{Z}_n,\mathbb{Z}_n)$ for odd values of $n$, we have the following corollary.
\begin{corollary}[\cite{V1} Corollary 14.]
The knot semigroup $M(T(2, n))$ of the diagram $T(2, n)$ for every positive $n$ is isomorphic to the strong alternating sum semigroup $\SAS(\mathbb{Z}_n,\mathbb{Z}_n)$.
\end{corollary} 
\subsection{Twist knots}
\begin{wrapfigure}[7]{r}[20pt]{5cm}
\includegraphics[keepaspectratio,width=4.5cm,bb=0 0 250 130]{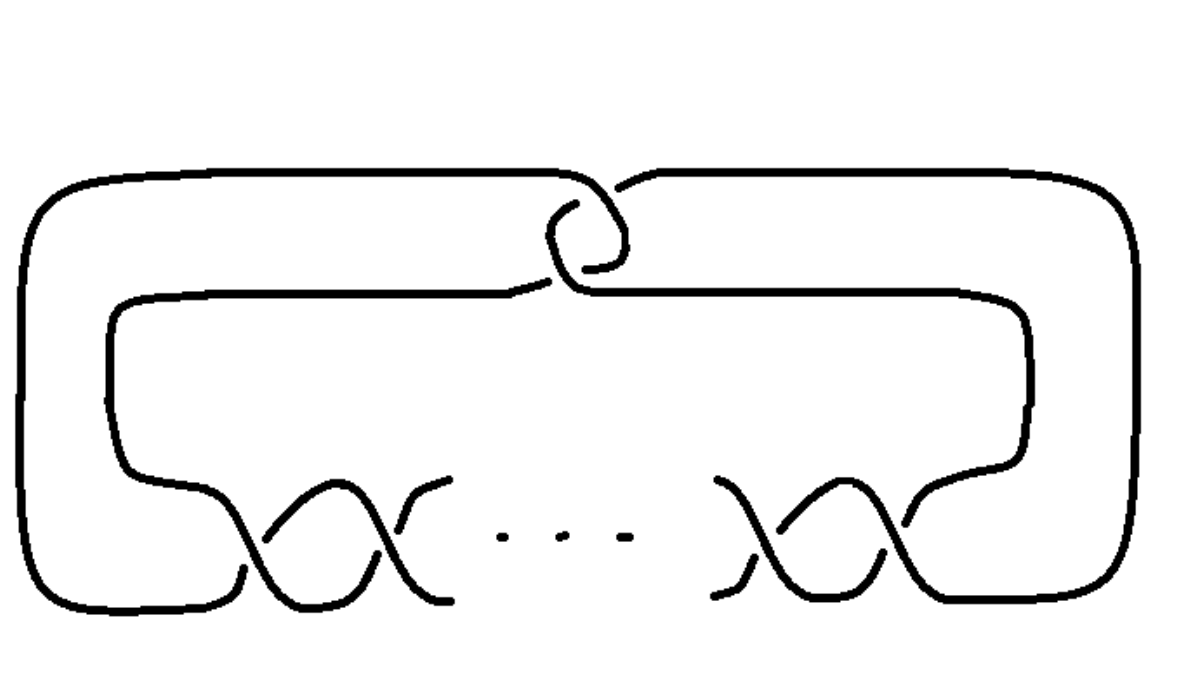}
\setlength{\parskip}{-3mm} 
\begin{center}Figure5 : Twist knots\end{center}
\end{wrapfigure}
A twist knot, which we shall denote $\mathfrak{tw}_n$ consists of $n$ clockwise half-twists and $2$ anticlockwise half-twists (Figure 5). We recall the knot semigroup $M(\mathfrak{tw}_n)$ of a knot diagram $\mathfrak{tw}_n$ proved by A. Vernitski \cite{V1}. The notation $[n+2]$ denote the set $\{0,1,\dots,n+2\}$.
\begin{theorem}[\cite{V1} Theorem 15.] \label{twistsemigroup}
The knot semigroup $M(\mathfrak{tw}_n)$ of the twist knot diagram $\mathfrak{tw}_n$ is isomorphic to the alternating sum semigroup $\AS(\mathbb{Z}_{2n+1},[n+2])$.
\end{theorem}

\section{Knot semigroups of double twist knots}
\subsection{A conjecture of the knot semigroups of  2-bridge knots}
The torus knots and the twist knots are the 2-bridge knots. Then we have the following conjecture by A. Vernitski (\cite{V1} Conjecture 23.).
\begin{conjecture} \label{bridgeconj}
The knot semigroup of the 2-bridge knot is isomorphic to an alternating sum semigroup.
\end{conjecture}
\subsection{Double twist knots}
To support the Conjecture \ref{bridgeconj} we prove that the knot semigroup of the double twist knot is isomorphic to an alternating sum semigroup.
\begin{center}
\includegraphics[keepaspectratio,width=4.5cm,bb=60 0 310 350]{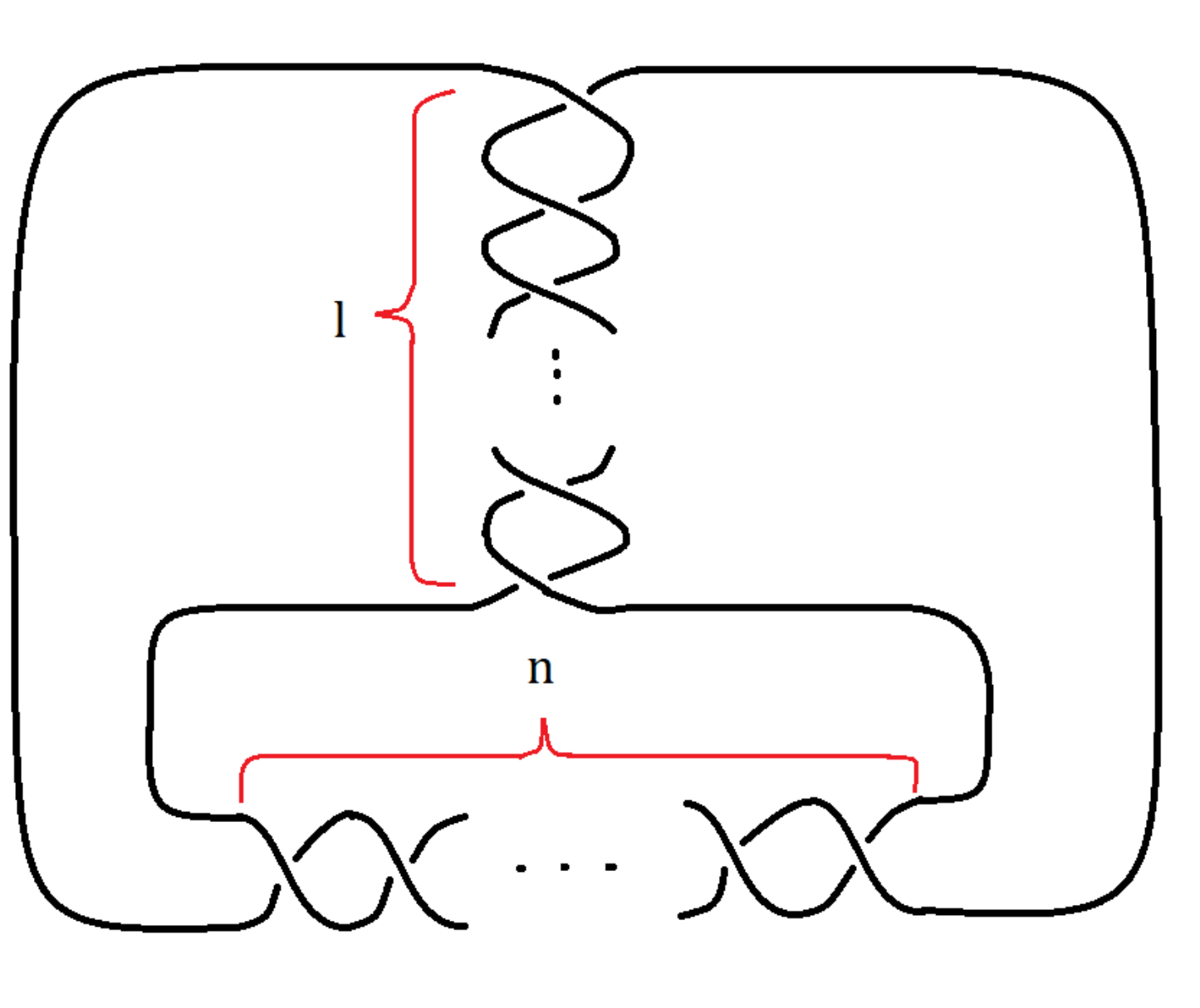}
\setlength{\parskip}{-3mm} 
\end{center}
\begin{center}Figure6 : Double twist knots\end{center}
A twist knot, which we shall denote $\mathfrak{dtw}_n^l$ consists of $n$ clockwise half-twists and $l$ anticlockwise half-twists, where $l,n$ indicate the number of crossing points (Figure 5). Then we have the following theorem.
\begin{theorem} \label{main}
Let $n,l \geq 1$ be integers. Suppose the integer $nl$ is an even integer. Then the knot semigroup $M(\mathfrak{dtw}_n^l)$ of the double twist knot diagram $\mathfrak{dtw}_n^l$ is isomorphic to the alternating sum semigroup
\begin{equation*}
\AS(\mathbb{Z}_{ln+1},\{0,1,\dots,n,0\cdot n+1,1\cdot n+1,\dots (l-1)\cdot n+1\}).
\end{equation*}
\end{theorem}
\begin{remark}
The double twist knot is the $2$-bridge knot $C(l,n)$.
\end{remark}
\begin{remark}
Let $l=2$ in Theorem \ref{main}. Then 
\begin{equation*}
\{0,1,\dots,0\cdot n+1,1\cdot n+1\}=[n+2].
\end{equation*}
Thus Theorem \ref{main} implies Theorem \ref{twistsemigroup}.
\end{remark}
\begin{remark}
Let $l=1$ in Theorem \ref{main}. Then
\begin{equation*}
\begin{split}
M(\mathfrak{dtw}_n^1)&\simeq \AS(\mathbb{Z}_{n+1},\{0,1,\dots,n,0\cdot n+1\}) \\
&\simeq \AS(\mathbb{Z}_{n+1},\{0,1,\dots,n\}) \\
&\simeq \AS(\mathbb{Z}_{n+1},\mathbb{Z}_{n+1}).
\end{split}
\end{equation*}
On the other hand, $\mathfrak{dtw}_n^1\simeq T(2,n+1)$ as knots. By Theorem \ref{torusthm}, 
\begin{equation*}
M(T(2,n+1))\simeq \AS(\mathbb{Z}_{n+1},\mathbb{Z}_{n+1}).
\end{equation*}
Thus Theorem \ref{main} holds in the case of $l=1$.
\end{remark}
\begin{remark}
We consider the following knot $C(m,l,n)$.
\begin{center}
\includegraphics[keepaspectratio,width=4.5cm,bb=100 0 350 157]{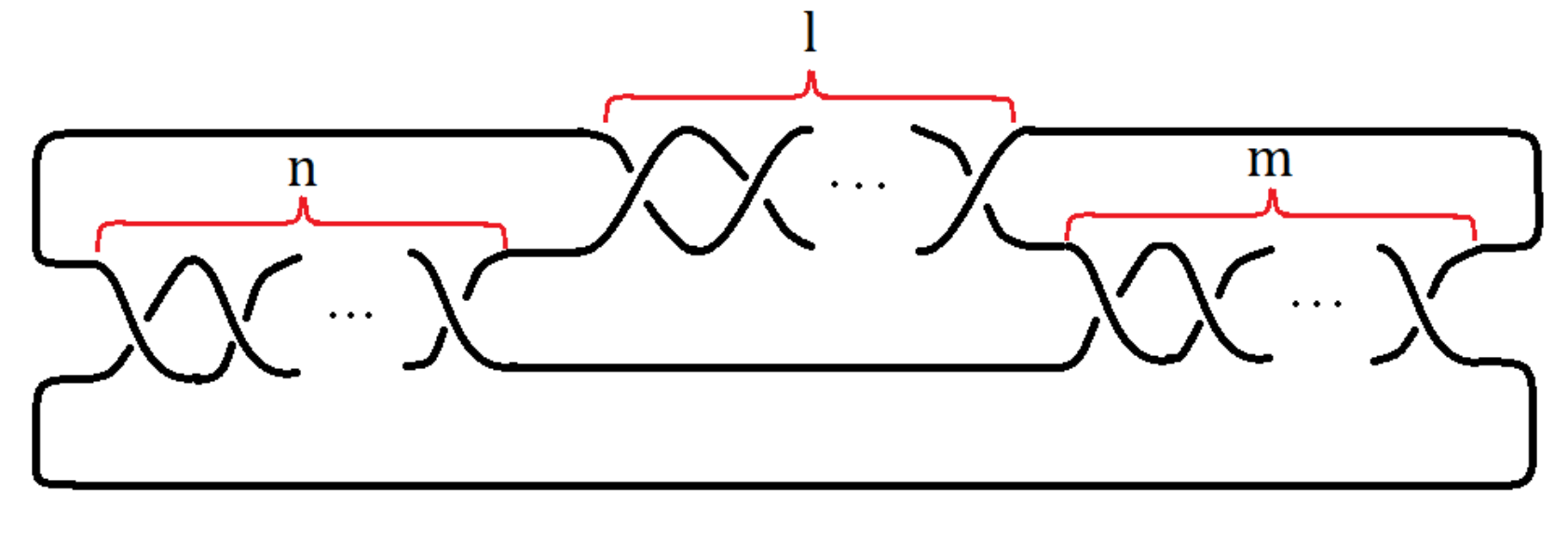}
\end{center}
Then we have the following conjecture.
\begin{conjecture}
Let $l,m,n\ge 1$ be integers. Suppose the integer $(ml+1)n+m$ is odd. Then the knot semigroup $M(C(m,l,n))$ of the knot diagram $C(m,l,n)$ is isomorphic to the alternating sum semigroup
\begin{equation*}
\AS(\mathbb{Z}_{(ml+1)n+m},\displaystyle\bigcup_{i=0}^{n+1}\{i\}\cup\displaystyle\bigcup_{j=0}^{l+1}\{j n+1\}\cup\displaystyle\bigcup_{k=0}^{m-1}\{(kl+1)n+k\}).
\end{equation*}
\end{conjecture}
\end{remark}
\subsection{Proof of the theorem \ref{main}}
Suppose that $A^+/\kappa$ is a knot semigroup, where $A$ is the set of arcs and $\kappa$ is a cancellative congruence on the free semigroup $A^+$ induced by the defining relations of the knot semigroup. Let $\sim$ be a congruence on $B^+$, where $B$ is an alphabet of the same size as $A$. We shall establish an isomorphism between $A^+/\kappa$ and $B^+/\sim$ by the following Lemma. 
\begin{lemma}[\cite{V1} Lemma 2.] \label{mainlemma}
Suppose A and B are sets. Consider a bijection $\phi : A\rightarrow B$. It induces an isomorphism between $A^+$ and $B^+$, which we shall denote $\phi^+$. Suppose a congruence $\kappa$ on $A^+$ and a congruence on $\sim$ on $B^+$ are such that for each $u, v\in A^+$ if $u \, \kappa \, v$ then $\phi(u)\sim\phi(v)$. Then $\phi$ induces a mapping from $A^+$ to $B^+$, which we shall denote by $\psi$. Moreover, $\psi$ is a homomorphism. Suppose a subset of $B^+$ exists, which we shall call the set of canonical words, such that in each class of $\sim$ there is exactly one canonical word and at least one word of each class of $\kappa$ is mapped by $\phi^+$ to a canonical word. Then $\psi$ is an isomorphism between $A^+/\kappa$ and $B^+/\sim$. 
\end{lemma} 
Let
\begin{equation*} 
A=\{a_0,\dots,a_n,a_{n+1},a_{2n+1},\dots,a_{(l-2)n+1},a_{(l-1)n+1}\}
\end{equation*}
be the set of arcs as in the following figure.
\begin{center}
\includegraphics[keepaspectratio,width=4.5cm,bb=60 0 310 350]{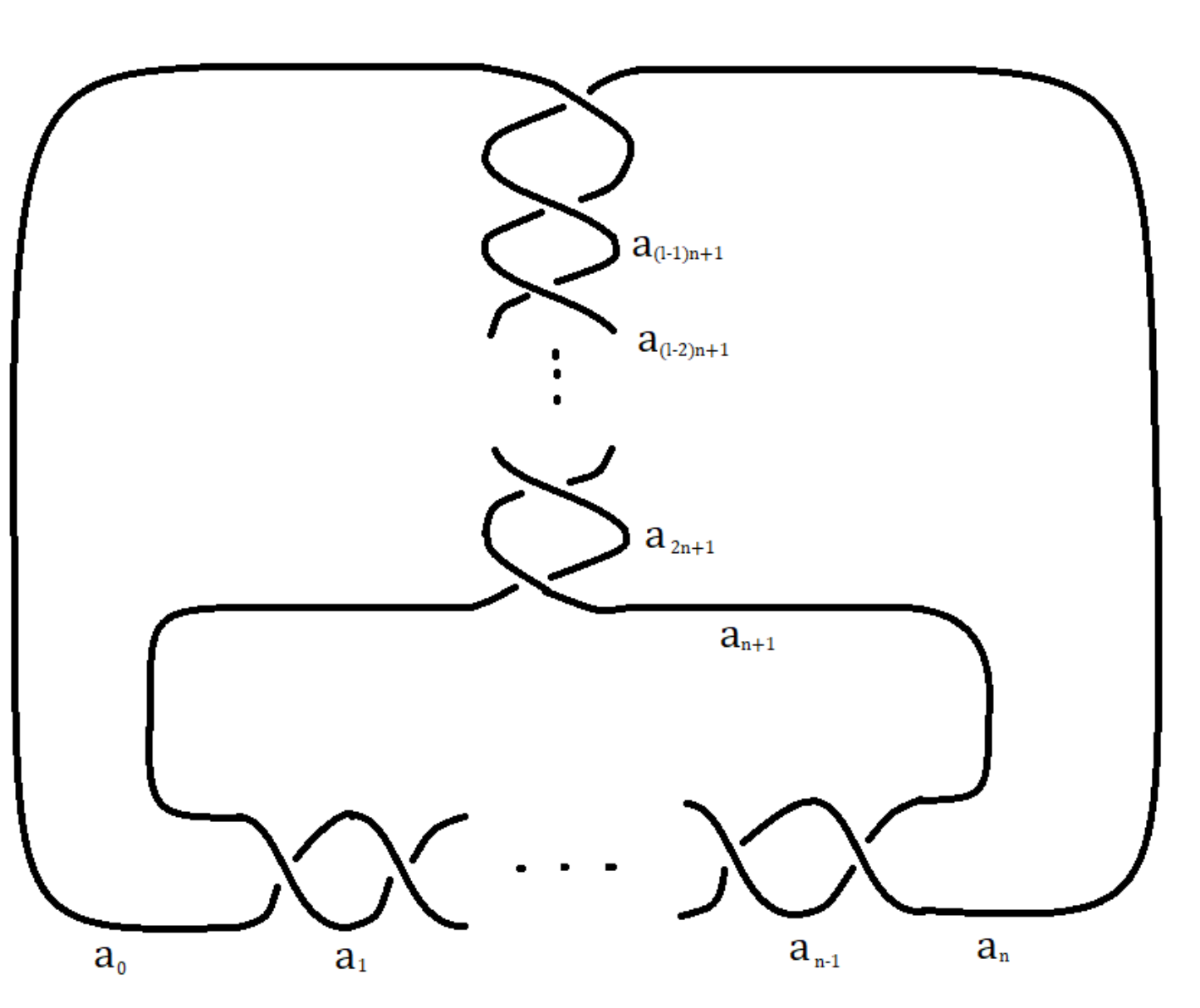}
\end{center}
Denote the set $\{0,1,\dots,n,0\cdot n+1,1\cdot n+1,\dots (l-1)\cdot n+1\}$ by $C_{n,l}$. Consider a mapping $\phi$ from $A$ to $C_{n,l}$ defined as $a_i\mapsto i$. It induces an isomorphism $A^+$ to $C_{n,l}^+$, which we shall denote by $\phi^+$. Then we have the following Lemma.
\begin{lemma} \label{lemma1}
The equality $a_ia_{i+j}=a_{i+k}a_{i+j+k}$ is true in $M(\mathfrak{dtw}_n^l)$ for all values of $i,j,k$ such that $0\le i\le i+j\le i+j+k \le n+1$. 
\end{lemma}
\begin{proof}
The relations in $M(\mathfrak{dtw}_n^l)$ are the equalities 
\begin{equation*}
a_{i-1}a_i=a_ia_{i+1},
\end{equation*}
and 
\begin{equation*}
a_ia_{i-1}=a_{i+1}a_i
\end{equation*}
for all $i=1,2,\dots,n$ (from the crossings at the bottom of the diagram), and the equalities 
\begin{equation*}
a_{(l-j-1)n+1}a_{(l-j)n+1}=a_{(l-j)n+1}a_{(l-j+1)n+1},
\end{equation*}
and 
\begin{equation*}
a_{(l-j)n+1}a_{(l-j-1)n+1}=a_{(l-j+1)n+1}a_{(l-j)n+1}
\end{equation*}
for all $j=0,1,\dots,l-1$ (from the crossings at the top of the diagram), where $a_{ln+1}=a_0$. Applying relations of the type $a_{i-1}a_i=a_ia_{i+1}$ repeatedly, we obtain $a_ia_{i+1}=a_{i+k}a_{i+1+k}$ for all values of $i,k$ such that $0\le i\le i+1\le i+1+k \le n+1$. Similarly, we can obtain $a_ia_{i-1}=a_{i+k}a_{i-1+k}$ for all values of $i,k$ such that $0\le i-1 \le i\le i+k \le n+1$. Consider 
\begin{equation*}
a_ia_{i+j}a_{i+j+1}=a_ia_{i+1}a_{i+2}=a_{i+2}a_{i+3}a_{i+2}=a_{i+2}a_{i+j+2}a_{i+j+1}.
\end{equation*}
Hence $a_ia_{i+j}=a_{i+2}a_{i+j+2}$ by the cancellative rule. This proves that $a_ia_{i+j}=a_{i+k}a_{i+j+k}$ for all values of  $i,j,k$ such that $0\le i\le i+j \le i+j+k \le n+1$ and even $k$. 

We shall prove that $a_0a_0=a_1a_1$.
\begin{enumerate}
\item[(1)]Suppose the integer $l$ is an even number. \\
Consider
\begin{equation*}
\begin{split}
a_{n+1}a_{\{l-(l-2)\}n+1}a_{\{l-(l-2)\}n+1}&=a_{\{l-(l-2)\}n+1}a_{\{l-(l-3)\}n+1}a_{\{l-(l-2)\}n+1} \\
&=a_{\{l-(l-2)\}n+1}a_{\{l-(l-2)\}n+1}a_{n+1} \\
&=a_{\{l-(l-2)\}n+1}a_{n+1}a_1 \\
&=a_{n+1}a_1a_1.
\end{split}
\end{equation*}
Hence we have $a_{\{l-(l-2)\}n+1}a_{\{l-(l-2)\}n+1}=a_1a_1$. \\
Next consider
\begin{equation*}
\begin{split}
a_{\{l-(l-3)\}n+1}a_{\{l-(l-4)\}n+1}a_{\{l-(l-4)\}n+1}&=a_{\{l-(l-4)\}n+1}a_{\{l-(l-5)\}n+1}a_{\{l-(l-4)\}n+1} \\
&=a_{\{l-(l-4)\}n+1}a_{\{l-(l-4)\}n+1}a_{\{l-(l-3)\}n+1} \\
&=a_{\{l-(l-4)\}n+1}a_{\{l-(l-3)\}n+1}a_{\{l-(l-2)\}n+1} \\
&=a_{\{l-(l-3)\}n+1}a_{\{l-(l-2)\}n+1}a_{\{l-(l-2)\}n+1}.
\end{split}
\end{equation*}
Hence $a_{\{l-(l-4)\}n+1}a_{\{l-(l-4)\}n+1}=a_{\{l-(l-2)\}n+1}a_{\{l-(l-2)\}n+1}$. \\
Since $l$ is even, $a_0a_0=a_1a_1$.
\item[(2)]Suppose the integer $l$ is an odd number. \\
Since $nl$ is even, $n$ is an even number. Consider 
\begin{equation*}
\begin{split}
a_{(l-1)n+1}a_0a_0&=a_0a_na_0 \\
&=a_0a_0a_{(l-1)n+1} \\
&=a_0a_{(l-1)n+1}a_{(l-2)n+1} \\
&=a_{(l-1)n+1}a_{(l-2)n+1}a_{(l-2)n+1}.
\end{split}
\end{equation*}
Hence $a_0a_0=a_{(l-2)n+1}a_{(l-2)n+1}$. Since this equation holds and $n$ is an even number,
\begin{equation*}
a_0a_0=a_{(l-2)n+1}a_{(l-2)n+1}=\dots=a_{n+1}a_{n+1}=a_1a_1.
\end{equation*} 
\end{enumerate} 
We shall prove that $a_0a_j=a_1a_{j+1}$ for all values of $j=0,1,\dots,n$. Let $j$ be an odd number. Consider
\begin{equation*}
a_0a_0a_j=a_{j-1}a_{j-1}a_j=a_{j-1}a_ja_{j+1}=a_0a_1a_{j+1}.
\end{equation*}
Hence $a_0a_j=a_1a_{j+1}$. \\ 
Let $j$ be even and positive. Consider 
\begin{equation*}
\begin{split}
a_0a_0a_j&=a_1a_1a_j  \,\,\,\,({\rm{Since}}\, a_0a_0=a_1a_1.)\\
&=a_{j-1}a_{j-1}a_j \\
&=a_{j-1}a_ja_{j+1} \\
&=a_{j-2}a_{j-1}a_{j+1} \\
&=a_0a_1a_{j+1}.
\end{split}
\end{equation*}
Hence $a_0a_j=a_1a_{j+1}$.\\
Now suppose $k$ is odd. If $i$ is even we have 
\begin{equation*}
\begin{split}
a_ia_{i+j}&=a_0a_j \\
&=a_1a_{j+1} \\
&=a_{i+1}a_{i+j+1} \\
&=a_{(i+1)+(k-1)}a_{(i+j+1)+(k-1)} \\
&=a_{i+k}a_{i+j+k}.
\end{split}
\end{equation*}
If $i$ is odd, we have
\begin{equation*}
\begin{split}
a_ia_{i+j}&=a_1a_{j+1} \\
&=a_0a_j \\
&=a_{i+1}a_{i+j+1} \\
&=a_{(i+1)+(k-1)}a_{(i+j+1)+(k-1)} \\
&=a_{i+k}a_{i+j+k}.
\end{split}
\end{equation*}
\end{proof}
\begin{lemma}
The equality $a_{in+1}a_{(i+j)n+1}=a_{(i+k)n+1}a_{(i+j+k)n+1}$ is true in $M(\mathfrak{dtw}_n^l)$ for all values of $i,j,k$ such that $0\le i\le i+j\le i+j+k\le l+1$.
\end{lemma}
\begin{proof}
Applying relations of the type 
\begin{equation*}
a_{(l-j-1)n+1}a_{(l-j)n+1}=a_{(l-j)n+1}a_{(l-j+1)n+1}
\end{equation*}
repeatedly, we obtain 
\begin{equation*}
a_{in+1}a_{(i+1)n+1}=a_{(i+k)n+1}a_{(i+1+k)n+1} 
\end{equation*}
for all values of $i,k$ such that $0\le i\le i+1 \le i+1+k \le l+1$. Similarly, we can obtain
\begin{equation*}
a_{in+1}a_{(i-1)n+1}=a_{(i+k)n+1}a_{(i-1+k)n+1}
\end{equation*}
for all values of $i,k$ such that $0\le i-1 \le i\le i+k \le l+1$. Consider
\begin{equation*}
\begin{split}
a_{in+1}a_{(i+j)n+1}a_{(i+j+1)n+1}&=a_{in+1}a_{(i+1)n+1}a_{(i+2)n+1} \\
&=a_{(i+2)n+1}a_{(i+3)n+1}a_{(i+2)n+1} \\
&=a_{(i+2)n+1}a_{(i+j+2)n+1}a_{(i+j+1)n+1}.
\end{split}
\end{equation*}
Hence $a_{in+1}a_{(i+j)n+1}=a_{(i+2)n+1}a_{(i+j+2)n+1}$. This proves that 
\begin{equation*}
a_{in+1}a_{(i+j)n+1}=a_{(i+k)n+1}a_{(i+j+k)n+1}
\end{equation*}
for all values of $i,j,k$ such that $0\le i\le i+j\le i+j+k\le l+1$ and even $k$. \\
We shall prove that $a_1a_1=a_{n+1}a_{n+1}$.
\begin{enumerate}
\item[(1)]Suppose $l$ is odd.\\
Since $n$ is even, $a_1a_1=a_{n+1}a_{n+1}$.
\item[(2)]Suppose $l$ is even, and $n$ is even. \\
Consider
\begin{equation*}
a_2a_1a_1=a_1a_0a_1=a_1a_1a_2=a_1a_2a_3=a_2a_3a_3.
\end{equation*}
Hence $a_1a_1=a_3a_3$. This proves that $a_1a_1=a_{n+1}a_{n+1}$. 
\item[(3)]Suppose $l$ is even and $n$ is odd. \\
Consider
\begin{equation*}
\begin{split}
a_na_{n+1}a_{n+1}&=a_{n-1}a_na_{n+1} \\
&=a_{n-1}a_{n-1}a_n \\
&=a_{n-1}a_{n-2}a_{n-1} \\
&=a_na_{n-1}a_{n-1}. \\
\end{split}
\end{equation*}
Hence $a_{n+1}a_{n+1}=a_{n-1}a_{n-1}$.Since $n$ is odd, $a_{n+1}a_{n+1}=a_0a_0$. Since $l$ is even, $a_0a_0=a_1a_1$. Thus $a_{n+1}a_{n+1}=a_1a_1$.
\end{enumerate}
We shall prove that $a_1a_{jn+1}=a_{n+1}a_{(j+1)n+1}$ for  all values $j=0,1,\dots,n$. Let $j$ be an odd number. Consider
\begin{equation*}
\begin{split}
a_1a_1a_{jn+1}&=a_{(j-1)n+1}a_{(j-1)n+1}a_{jn+1} \\
&=a_{(j-1)n+1}a_{jn+1}a_{(j+1)n+1} \\
&=a_1a_{n+1}a_{(j+1)n+1}.
\end{split}
\end{equation*}
Hence $a_1a_{jn+1}=a_{n+1}a_{(j+1)n+1}$. \\
Let $j$ be even and positive. Consider
\begin{equation*}
\begin{split}
a_1a_1a_{jn+1}&=a_{n+1}a_{n+1}a_{jn+1} \\
&=a_{(j-1)n+1}a_{(j-1)n+1}a_{jn+1} \\
&=a_{(j-1)n+1}a_{jn+1}a_{(j+1)n+1} \\
&=a_{(j-2)n+1}a_{(j-1)n+1}a_{(j+1)n+1} \\
&=a_1a_{n+1}a_{(j+1)n+1}.
\end{split}
\end{equation*}
Hence $a_1a_{jn+1}=a_{n+1}a_{(j+1)n+1}$. \\
Suppose $k$ is odd. If $i$ is even we have
\begin{equation*}
\begin{split}
a_{in+1}a_{(i+j)n}&=a_1a_{jn+1} \\
&=a_{n+1}a_{(j+1)n+1} \\
&=a_{(i+1)n+1}a_{(i+j+1)n+1} \\
&=a_{\{(i+1)+(k-1)\}n+1}a_{\{(i+j+1)+(k-1)\}n+1} \\
&=a_{(i+k)n+1}a_{(i+j+k)n+1}.
\end{split}
\end{equation*}
If $i$ is odd, we have
\begin{equation*}
\begin{split}
a_{in+1}a_{(i+j)n}&=a_{n+1}a_{(j+1)n+1} \\
&=a_1a_{jn+1} \\
&=a_{(i+1)n+1}a_{(i+j+1)n+1} \\
&=a_{\{(i+1)+(k-1)\}n+1}a_{\{(i+j+1)+(k-1)\}n+1} \\
&=a_{(i+k)n+1}a_{(i+j+k)n+1}.
\end{split}
\end{equation*}
\end{proof}
\begin{lemma} \label{trans2}
The equality $a_{pn+1}a_qa_{rn+1}=a_{(p+1)n+1}a_qa_{(r-1)n+1}$ is true in $M(\mathfrak{dtw}_n^l)$ for all values $p,q,r$ such that $0\le p\le n-1$, $1\le r\le l$, and $q\not\in \{0,2n+1,\dots,(l-1)n+1\}$. 
\end{lemma}
\begin{proof}
Consider
\begin{equation*}
\begin{split}
a_{(p+1)n+1}a_{pn+1}a_qa_{rn+1}&=a_na_0a_qa_{rn+1} \\
&=a_na_{n-q+1}a_{n+1}a_{rn+1} \\
&=a_na_{n-q+1}a_1a_{(r-1)n+1} \\
&=a_na_na_qa_{(r-1)n+1} \\
&=a_{n+1}a_{n+1}a_qa_{(r-1)n+1} \\
&=a_{(p+1)n+1}a_{(p+1)n+1}a_qa_{(r-1)n+1}.
\end{split}
\end{equation*}
Hence $a_{pn+1}a_qa_{rn+1}=a_{(p+1)n+1}a_qa_{(r-1)n+1}$.
\end{proof}
\begin{lemma} \label{trans1}
The equality $a_pa_qa_r=a_{p+1}a_qa_{r-1}$ is true in $M(\mathfrak{dtw}_n^l)$ for all values $p,q,r$ such that $p\in\{0,1,\dots,n\}$, $r\in\{1,2,\dots,n+1\}$, $q\in\{0,2n+1,\dots,(l-1)n+1\}$.
\end{lemma}
\begin{proof}
Suppose $q=0$. Then
\begin{equation*}
a_pa_0a_r=a_{p+1}a_1a_r=a_{p+1}a_0a_{r-1}.
\end{equation*}
Suppose $q\in\{2n+1,\dots,(l-1)n+1\}$. Consider
\begin{equation*}
\begin{split}
a_{p+1}a_pa_{kn+1}a_r&=a_1a_0a_{kn+1}a_r \\
&=a_1a_{(l-k)n+1}a_1a_r \\
&=a_1a_{(l-k)n+1}a_0a_{r-1} \\
&=a_{kn+1}a_0a_0a_{r-1} \\
&=a_{kn+1}a_1a_1a_{r-1} \\
&=a_{kn+1}a_{(k-1)n+1}a_{(k-1)n+1}a_{r-1} \\
&=a_{n+1}a_1a_{(k-1)n+1}a_{r-1} \\
&=a_{2n+1}a_{n+1}a_{(k-1)n+1}a_{r-1} \\
&=a_{n+1}a_{n+1}a_{kn+1}a_{r-1} \\
&=a_{p+1}a_{p+1}a_{kn+1}a_{r-1}.
\end{split}
\end{equation*}
Hence $a_pa_qa_r=a_{p+1}a_qa_{r-1}$.
\end{proof}

Canonical words in $C_{n,l}^+$ will be defined as words of the form $000^{t-2}$ or $c00^{t-2}$ or $0c0^{t-2}$ or $dc0^{t-2}$, where $t\ge 2$ is the length of the word and $c\in \{1,2,\dots,n\}$, $d\in \{2n+1,3n+1,\dots,(l-1)n+1\}$. 

Consider a non-negative integer valued parameter $\pi(w)$ of a word $w$ in $C_{n,l}^+$, which is $0$ if the first two entries in $w$ are $00$ or $c0$ or $0c$ or $dc$ for some $c\in \{1,\dots,n\}$, $d\in \{2n+1,\dots,(l-1)n+1\}$ and which is $1$ otherwise. Define the defect of a word $w=b_1b_2\dots b_t$ in $C_{n,l}^+$ as a word $\pi(w)b_3\dots b_t$. Defects are assumed to be ordered antilexicographically.  
\begin{lemma} \label{lem2}
A word in $C_{n,l}^+$ is canonical if and only if its defect is a word consisting of $0$s.
\end{lemma}
\begin{proof}
The result follows from the form of canonical word.
\end{proof}
\begin{lemma}
Let $u$ be a word in $A^+$. Unless the defect of $\phi(u)$ is a word consisting of $0$s, there is a word $v$ in $A^+$ such that $u=v$ in $M(\mathfrak{dtw}_n^l)$ and the defect of $\phi(v)$ is less than the defect of $\phi(u)$. 
\end{lemma}
\begin{proof}
Suppose the defect of $\phi(u)$ has a non-zero entry at a position which is not the first one. This means that at some position $d\ge 3$ there is a non-zero entry $r$ in $\phi(u)$. Let $u=u^{\prime}a_pa_qa_ru^{\prime\prime}$, where $u^{\prime}, u^{\prime\prime}\in A^+$ and $p,q,r\in C_{n,l}$ with $r\neq 0$.
\begin{enumerate}
\item[(1)]Suppose $q\not\in\{0,2n+1,3n+1,\dots,(l-1)n+1\}$.
 \begin{enumerate}
 \item[1.]If $r\not\in\{2n+1,\dots,(l-1)n+1\}$, then we define 
 \begin{equation*}
 v=u^{\prime}a_pa_{q-1}a_{r-1}u^{\prime\prime}. 
 \end{equation*}
 \item[2.]If $r\in\{2n+1,3n+1,\dots,(l-1)n+1\}$.
  \begin{enumerate}
  \item[$\cdot$]If $p\not\in\{0,2n+1,\dots,(l-1)n+1\}$, then we define 
  \begin{equation*}
  v=u^{\prime}a_{p-1}a_{q-1}a_ru^{\prime\prime}.
  \end{equation*}
  \item[$\cdot$]If $p\in \{0,2n+1,\dots,(l-1)n+1\}$, then  define   
  \begin{equation*}
   v=u^{\prime}a_{(k^{\prime}+1)n+1}a_qa_{(k-1)n+1}u^{\prime\prime}, 
   \end{equation*}
   where $p=k^{\prime}n+1, r=kn+1$. The words $u$ and $v$ are equal by the Lemma \ref{trans2}.
  \end{enumerate}
 \end{enumerate}
\item[(2)]Suppose $q\in \{0,2n+1,\dots,(l-1)n+1\}$.
 \begin{enumerate}
 \item[1.]If $p\in\{2n+1,\dots,(l-1)n+1\}$ or $r\in\{2n+1,\dots,(l-1)n+1\}$, then we define
 \begin{equation*}
 v=u^{\prime}a_{(k^{\prime}-1)n+1}a_{(k-1)n+1}a_ru^{\prime\prime}, 
 \end{equation*}
 where $p=k^{\prime}n+1,q=kn+1$, or 
 \begin{equation*}
 v=u^{\prime}a_pa_{(k-1)n+1}a_{(k^{\prime\prime}-1)n+1}u^{\prime\prime},
 \end{equation*}
 where $q=kn+1, r=k^{\prime\prime}n+1$.
 \item[2.]If $p,r\not\in\{2n+1,3n+1\dots,(l-1)n+1\}$, then we define 
 \begin{equation*}
 v=u^{\prime}a_{p+1}a_qa_{r-1}u^{\prime\prime}. 
 \end{equation*}
 By the Lemma \ref{trans1}, $u=v$.
 \end{enumerate}
\end{enumerate}
In each case, $u=v$ in $M(\mathfrak{dtw}_n^l)$, and the defect of $\phi(v)$ is less than the defect of $\phi(u)$.

Suppose the defect of $\phi(u)$ has a non-zero entry only at the first position. Let $u=a_pa_qu^{\prime}$, where $u^{\prime}\in A^+$ and $p,q\in C_n^l$.
\begin{enumerate}
\item[(1)]If $p,q\in\{1,2,\dots,n+1\}$, let $m=\min(p,q)$, then define 
\begin{equation*}
v=a_{p-m}a_{q-m}u^{\prime}.
\end{equation*}
\item[(2)]If $p,q\in\{2n+1,\dots,(l-1)n+1\}$, let $p=kn+1,q=k^{\prime}n+1$ and $m=\min(k,k^{\prime})$. If $m=k^{\prime}$, then we define 
\begin{equation*}
v=a_{(k-m)n+1}a_{(k^{\prime}-m)n+1}u^{\prime}.
\end{equation*}
If $m=k$, then we use the following case $(3)$.
\item[(3)]If $p\in\{1,2,\dots,n\}, q\in\{2n+1,\dots,(l-1)n+1\}$, then consider
\begin{equation*}
\begin{split}
a_pa_{kn+1}a_{n-p+1}&=a_0a_{kn+1}a_{n+1} \\
&=a_{(l-k)n+1}a_1a_{n+1} \\
&=a_{(l-k+1)n+1}a_{n+1}a_{n+1} \\
&=a_{(l-k+1)n+1}a_{n-p+1}a_{n-p+1}.
\end{split}
\end{equation*}
Hence $a_pa_{kn+1}=a_{(l-k+1)n+1}a_{n-p+1}$. Then we define 
\begin{equation*}
v=a_{(l-k+1)n+1}a_{n-p+1}u^{\prime}, 
\end{equation*}
where $q=kn+1$.
\item[(4)]If $p\in\{2n+1\dots,(l-1)n+1\}, q=n+1$, then we define 
\begin{equation*}
v=a_{(k-1)n+1}a_1u^{\prime}, 
\end{equation*}
where $p=kn+1$.
\item[(5)]If $p=n+1, q\in\{2n+1,\dots,(l-1)n+1\}$, then we consider
\begin{equation*}
a_{n+1}a_{kn+1}=a_{(l-k+1)n+1}a_0=a_{(l-k+2)n+1}a_n.
\end{equation*}
Then we define 
\begin{equation*}
v=a_{(l-k+2)n+1}a_nu^{\prime}.
\end{equation*}
\item[(6)]If $p\in\{n+1,2n+1,\dots,(l-1)n+1\}, q=0$, then we consider 
\begin{equation*}
a_{kn+1}a_0=a_1a_{(l-k)n+1}=a_{(k+1)n+1}a_n.
\end{equation*}
Then we define 
\begin{equation*}
v=a_{(k+1)n+1}a_nu^{\prime}, 
\end{equation*}
where $p=kn+1$.
\item[(7)]If $p=0, q\in\{n+1,2n+1,\dots,(l-1)n+1\}$, then consider
\begin{equation*}
a_0a_{kn+1}=a_{(l-k)n+1}a_1
\end{equation*}
Then we define 
\begin{equation*}
v=a_{(l-k)n+1}a_1u^{\prime}, 
\end{equation*}
where $q=kn+1$.
\end{enumerate}
In each case, $u=v$ in $M(\mathfrak{dtw}_n^l)$, and the defect $\phi(v)$ is a word consisting of $0$s. 
\end{proof}
Then we have the following corollary.
\begin{corollary}
Every word in $A^+$ is equal in $M(\mathfrak{dtw}_n^l)$ to a word in $A^+$ which is mapped by $\phi$ to a word with a defect consisting of $0$s.
\end{corollary} \label{cor1}
{\noindent{\it{Proof of theorem \ref{main}.}}} The relations in $M(\mathfrak{dtw}_n^l)$ are listed in the proof of Lemma \ref{lemma1}. For each relation $u=v$ the words $\phi(u)$ and $\phi(v)$ have the same length and the same alternating sum calculated in $\mathbb{Z}_{ln+1}$. Thus by Lemma \ref{mainlemma}, $\phi^+$ induces a homomorphism $\psi : M(\mathfrak{dtw}_n^l)\rightarrow C_{n,l}^+/\sim$. \\
Consider two canonical words $u, v$ which are $\sim$ equivalent. We shall show that each class of $\sim$ contains at most one canonical words.
\begin{enumerate}
\item[(1)]Suppose their alternating sums are both $0$.
 \begin{enumerate}
 \item[1.]If $u, v$ are form of $c00^{t-2}$ or $0c0^{t-2}$, where $c\in\{1,2,\dots,n\}$, then since the canonical word can have at most one non-zero entry, both words consist only of $0$s and, therefore are equal.
 \item[2.]If $u$ or $v$ is form of $dc0^{t-2}$, where $c\in\{1,2,\dots,n\}$, $d\in\{2n+1,dots,(l-1)n+1\}$, then the alternating sum of $dc0^{t-2}$ is $d-c$. If $d\neq 0$ or $c\neq 0$, then since $d-c=0$, this contradicts $c\in\{1,2,\dots,n\}$, $d\in\{2n+1,\dots,(l-1)n+1\}$. Therefore both words $u,v$ consists only of $0$s, and are equal.
 \end{enumerate}
\item[(2)]Suppose two canonical words share the same non-zero alternating sum.
 \begin{enumerate}
 \item[1.]If $u=c_100^{t-2}$, $v=c_200^{t-2}$, then since both alternating sum is the same, $c_1=c_2$. Thus $u=v$.
 \item[2.]If $u=0c_10^{t-2}$, $v=0c_20^{t-2}$, then $u=v$ by the same reason of $1$.
 \item[3.]If $u=c_100^{t-2}$, $v=0c_20^{t-2}$, then $c_1=-c_2$ in $\mathbb{Z}_{nl+1}$. Since $c_1,c_2\in\{1,2,\dots,n\}$, this case is impossible. 
 \item[4.]If $u=c_100^{t-2}$, $v=0c_20^{t-2}$, then $c_1=d_2-c_2$ in $\mathbb{Z}_{ln+1}$. Since $c_1,c_2\in\{1,2,\dots,n\}$, $d\in\{2n+1,\dots,(l-1)n+1\}$, this case is impossible ($c_1+c_2\le 2n, d_2\ge 2n$).
 \item[5.]If $u=0c_10^{t-2}$, $v=d_2c_20^{t-2}$, then $-c_1=d_2-c_2$ in $\mathbb{Z}_{ln+1}$. If $c_2-c_1>0$, this case is impossible. If $c_2-c_1>0$, this case is also impossible ($ln+1+c_2-c_1>(l-1)n+1$, $d_2\le (l-1)n+1$). 
 \item[6.]If $u=d_1c_10^{t-2}$, $v=d_2c_20^{t-2}$, then $d_1-c_1=d_2-c_2$ in $\mathbb{Z}_{ln+1}$. Since $c_1,c_2\in\{1,2,\dots,n\}$, $d_1, d_2\in\{2n+1,\dots,(l-1)n+1\}$, this case is impossible. 
 \end{enumerate}
\end{enumerate}
Thus each class of $\sim$ contains at most one canonical words. \\
Consider a words $w\in C_{n,l}^+$ which has a length $t$ and alternating sum $s$. We shall show that each class of $\sim$ contains at least one canonical word.
\begin{enumerate}
\item[(1)]If $s\in\{0,1,\dots,n\}$, then canonical words $s00^{t-2}$ is $\sim$ equivalent to $w$.
\item[(2)]If $s\in\{(l-1)n+1,\dots,ln+1\}$, let $q=-s$. Then $0q0^{t-2}$ is $\sim$ equivalent to $w$.
\item[(3)]If $s\in\{n+1,\dots,(l-1)n\}$, then $w$ is $\sim$ equivalent to $d(-c)0^{t-2}$ for some $c\in\{1,2,\dots,n\}$, $d\in\{2n+1,\dots,(l-1)n+1\}$.
\end{enumerate}
Thus each class of $\sim$ contains at least one canonical word. \\
By Corollary \ref{cor1} and Lemma \ref{lem2}, each word in $A^+$ is equal in $M(\mathfrak{dtw}_n^l)$ to a word mapped by $\phi$ to a canonical word. Now Theorem \ref{main} follows from Lemma \ref{mainlemma}.

\section{Growth function of knot semigroups}
\subsection{Growth function and skew growth function}
As a cororally of Theorem \ref{torusthm}, \ref{main}, we can compute growth functions of some knot semigroups. First we explain growth functions of a monid. 

Let $M$ be a monid with unit $1$. For $u,v\in M$, we denote
\begin{equation*}
u\mid_{l}v
\end{equation*}
if there exists an element $x\in M$ such that $v=ux$. We define an equivalence relation on $M$ by putting $u\sim v$ if and only if $u\mid_{l}v$ and $v\mid_{l}u$. We consider a monid $M/\sim$.
\begin{definition}
A discrete degree map on a monoid $M$ is a map
\begin{equation*}
{\rm{deg}}:M\longrightarrow\mathbb{R}_{\geq 0}
\end{equation*}
such that
\begin{enumerate}
\item[(1)]$\deg(u)=0$ if and only if $u\sim 1$,
\item[(2)]$\deg(uv)=\deg(u)+\deg(v)$ for any $u,v\in M$,
\item[(3)]$\sharp\{u\in M/\sim\mid\deg(u\leq r)\}<\infty$ for any $r\in\mathbb{R}_{> 0}$.
\end{enumerate}
\end{definition}
Let 
\begin{equation*}
(M/\sim)_{d}=\{u\in M/\sim\mid\deg (u)=d\}.
\end{equation*}
Then we define the growth function of the monoid $M$ by 
\begin{equation*}
P_{M,\deg}(t):=\displaystyle\sum_{u\in M/\sim}t^{\deg(u)}=\displaystyle\sum_{d\in\mathbb{R}_{>0}}\sharp((M/\sim)_d)t^d.
\end{equation*}
A skew growth function $N_{M,\deg}(t)$ is defined in \cite{S} Section 4.2. The we have the following theorem. Let
\begin{equation*}
R_{\mathbb{Z}}=\left\{\displaystyle\sum^{\infty}_{n=0}\mid 
                   \begin{matrix}
                   a_n\in\mathbb{Z}(\forall n\in\mathbb{Z}_{\geq 0}), \text{ and } \{d_n\}_{n\in\mathbb{Z}_{\geq 0}} \text{ is } \\
                   \text{ a sequence in } \mathbb{R}_{\geq 0} \text{ divergent to } +\infty
                   \end{matrix} 
                 \right\}.
\end{equation*}
\begin{theorem}[\cite{S}] \label{inverse}
Let $M$ be a cancellative monoid equipped with a discrete degree map. Then we have the following formula in the ring $R_{\mathbb{Z}}$.
\begin{equation*}
P_{M,\deg}(t)\cdot N_{M,\deg}(t)=1.
\end{equation*} 
\end{theorem}
\subsection{Growth function of knot semigroups}
In case of torus knots and double twist knot, we have the following corollary. In the case of knot semigroups we have
\begin{equation*}
\begin{split}
&M(K)/\sim =M(K), \\
&\deg=\mid\cdot\mid.
\end{split}
\end{equation*} 
\begin{corollary}
Let $b$ be an odd integer and $ml$ even integer. Then we have
\begin{equation}
\begin{split}
&P_{M(T(2,n)),\mid\cdot\mid}=\dfrac{(n-1)t+1}{1-t}, \\
&N_{M(T(2,n)),\mid\cdot\mid}=\dfrac{1-t}{(n-1)t+1},
\end{split}
\end{equation}
\begin{equation}
\begin{split}
&P_{\mathfrak{dtw}_m^l,\mid\cdot\mid}=\dfrac{1+(m+l-1)t+(ml-m-l+1)t^2}{1-t}, \\ &N_{\mathfrak{dtw}_m^l,\mid\cdot\mid}=\dfrac{1-t}{1+(m+l-1)t+(ml-m-l+1)t^2}.
\end{split}
\end{equation}
\end{corollary}
\begin{proof}
By Theorem \ref{torusthm}, we have
\begin{equation*}
(M(T(2,n)))_0=1,(M(T(2,n)))_d=n (d\geq 1).
\end{equation*}
Thus
\begin{equation*}
\begin{split}
P_{M(T(2,n)),\mid\cdot\mid}&=1+n\displaystyle\sum_{i=0}^{\infty}t^i\\
&=1+\dfrac{tn}{1-t}\\
&=\dfrac{(n-1)t+1}{1-t}.
\end{split}
\end{equation*}
$N_{M(T(2,n)),\mid\cdot\mid}$ can be computed by Theorem \ref{inverse}.

By Theorem \ref{main}, we have
\begin{equation*}
\begin{split}
&(M(\mathfrak{dtw}_m^l))_0=1,(M(\mathfrak{dtw}_m^l))_1=m+l,  \\
&(M(\mathfrak{dtw}_m^l))_d=lm+1(d\geq 2).
\end{split}
\end{equation*}
Thus 
\begin{equation*}
\begin{split}
P_{M(\mathfrak{dtw}_m^l),\mid\cdot\mid}&=1+(m+l)t+(lm+1)\displaystyle\sum_{i=2}^{\infty}t^i\\
&=1+(m+l)t+\dfrac{(lm+1)t^2}{1-t}\\
&=\dfrac{1+(m+l-1)t+(lm-m-l+1)t^2}{1-t}.
\end{split}
\end{equation*}
$N_{M(\mathfrak{dtw}_m^l),\mid\cdot\mid}$ can be computed by Theorem \ref{inverse}.
\end{proof}

\section{Link invariants constructed by Gelfand-Kirillov dimensions}
In this section, we consider the growth of semigroup algebras of knot semigroups. To investigate the growth, we use a Gelfand-Kirillov dimension. As an application and the main result, we construct a link invariant.

\subsection{Gelfand-Kirillov dimensions}
Let $k$ be a field and $A$ be a finitely generated algebra over $k$. Let $V$ be a subspace of A, and we denote by $V^n$ the subspace spanned by all products of elements of $V$ of length $n$. The subspace $V$ is called generating subspace if $V$ is finite-dimensional subspace of $A$ which generate $A$ as algebra, and contains $1$. For a generating subspace $V$, we define the function $f_V(n)$ such that
\begin{equation*}
f_V(n)=\dim_kV^n.
\end{equation*}
An algebra $A$ is said to have polynomial growth if there are positive real numbers $c,r$ such that
\begin{equation*}
f_V(n)\leq cn^r
\end{equation*}
for all $n$. Then this definition is independent of the choice of generating spaces.
\begin{lemma}
Let $A$ be a finitely generated $k$-algebra, let $V, W$ be generating subspaces. If $f_V(n)\leq cn^r$ for all $n$, then there is a $c^{\prime}$ such that $f_W(n)\leq c^{\prime}n^r$ for all $n$.
\end{lemma}
\begin{proof}
Suppose $f_V(n)\leq cn^r$. Since $A=\displaystyle\cup V^n$ and $W$ is finite-dimensional, $W\subset V^s$ for some $s$. Then $W^n\subset V^{sn}$. Hence
\begin{equation*}
f_W(n)\leq f_V(sn)\leq cs^rn^r=c^{\prime}n^r.
\end{equation*}
\end{proof}
\begin{definition}
The Gelfand-Kirillov dimension of an algebra $A$ with polynomial growth is defined by
\begin{equation*}
{\rm{GK}}\dim(A)=\inf\{r\mid f_V(n)\leq cn^r\}.
\end{equation*}
If $A$ does not have polynomial growth then we define ${\rm{GK}}\dim(A)=\infty$.
\end{definition}
\begin{example}
(1)Let $A=k[x_1,\dots,x_d]$. If we let V be the space spanned by $\{1,x_1,\dots,x_d\}$, then $V^n$ is the space of polynomials of degree $\leq n$. The dimension of this space is $\begin{pmatrix} n+d \\ d\end{pmatrix}$. Thus ${\rm{GK}}\dim(A)=d$. \\
(2)Let $A=k\langle x_1,\dots,x_d\rangle$ be the noncommutative polynomial algebra. Let V be the space spanned by $\{1,x_1,\dots,x_d\}$. The dimension of this space is $d^n+d^{n-1}+\dots+d+1$. Thus ${\rm{GK}}\dim(A)=\infty$.
\end{example}
\subsection{Link invariants}
We can prove that the Gelfand-Kirillov dimension of semigroup algebra of knot semigrops is a link invariant. Let $L$ be a link. Then we let $k(M(L))$ be a semigroup algebra of a knot semigroup of $L$.
\begin{theorem}
Let $L_1,L_2$ be links. If $L_1\simeq L_2$ then
\begin{equation*}
{\rm{GK}}\dim(k(M(L_1)))={\rm{GK}}\dim(k(M(L_2))).
\end{equation*}
\end{theorem}
\begin{proof}
For a link $L$, we prove ${\rm{GK}}\dim(k(M(L)))$ is invariant under Reidemeister move ${\rm{I,I\hspace{-.1em}I,I\hspace{-.1em}I\hspace{-.1em}I}}$. ${\rm{GK}}\dim(k(M(L)))$ is invariant under Reidemeister move ${\rm{I}}$ since $M(L)$ is. Let $L$,$L^{\prime}$ be the same links, except in the neighborhood of a point where they are as shown in Figure.
\begin{center}
\includegraphics[keepaspectratio,width=5cm,bb=0 0 250 230]{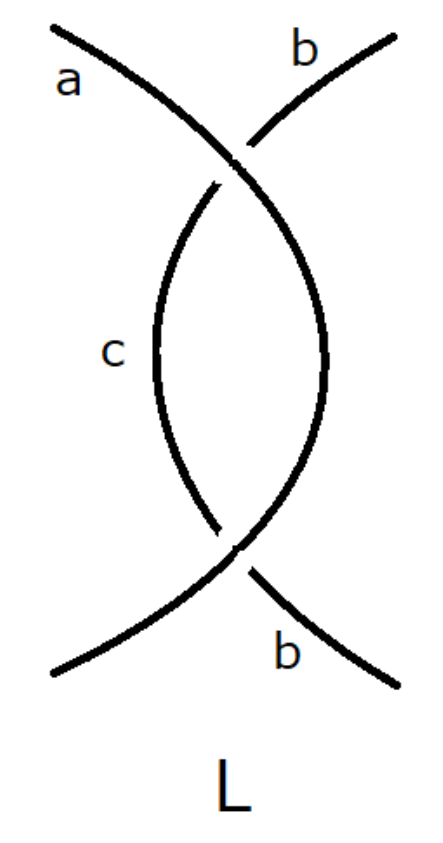}
\includegraphics[keepaspectratio,width=5cm,bb=0 0 250 230]{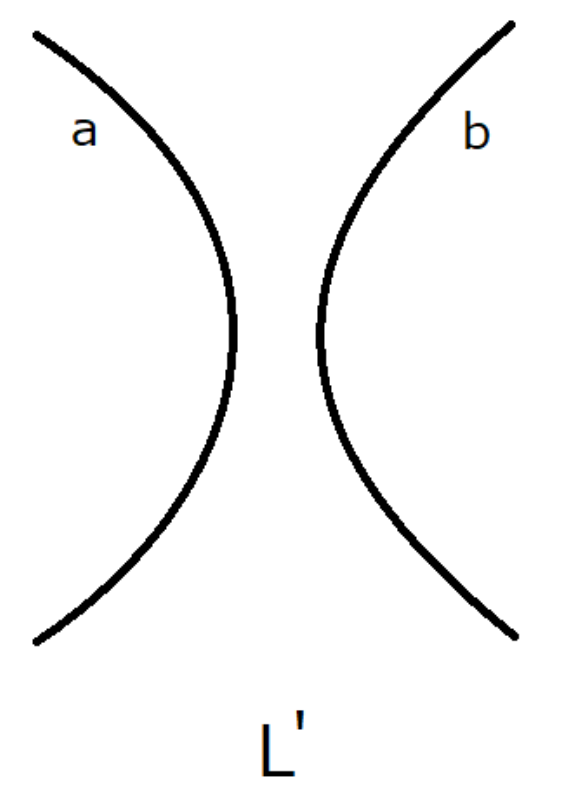}
\end{center}
We can assume ${\rm{GK}}\dim(k(M(L))),{\rm{GK}}\dim(k(M(L^{\prime})))<\infty$. Then
\begin{equation*}
\begin{split}
&k(M(L))\simeq k\langle x_1,\dots,x_m,a,b,c\rangle/\langle I_{a,b}\cup\left\{\begin{matrix}ba-ac \\ ab-ca \end{matrix}\right\}\rangle, \\
&k(M(L^{\prime}))\simeq k\langle x_1,\dots,x_m,a,b\rangle/\langle I_{a,b}\rangle,
\end{split}
\end{equation*}
where $x_1,\dots,x_m$ are labels of arcs except $a,b,c$, and $I_{a,b}$ is a set of relations except $ba-ac,ab-ca$. Let 
\begin{equation*}
\begin{split}
&V=\langle x_1,\dots,x_m,a,b,c\rangle/\langle I_{a,b}\cup\left\{\begin{matrix}ba-ac \\ ab-ca \end{matrix}\right\}\rangle, \\
&V^{\prime}=\langle x_1,\dots,x_m,a,b\rangle/\langle I_{a,b}\rangle.
\end{split}
\end{equation*}
Then we have the following exact sequences for $n\geq 2$.
\begin{equation*}
\begin{split}
0\rightarrow \langle c-1\rangle^n \rightarrow V^n \rightarrow \left(V^n/\langle c-1\rangle \right)^n\rightarrow 0, \\
0\rightarrow \langle b-1\rangle^n \rightarrow V^n \rightarrow \left(V^n/\langle b-1\rangle \right)^n\rightarrow 0. 
\end{split}
\end{equation*}
Since 
\begin{equation*}
\begin{split}
&f_V(n)=\dim V^n=\dim (V/\langle c-1\rangle)^n+\langle c-1\rangle^n, \\
&f_{V^{\prime}}(n)=\dim {V^{\prime}}^n=\dim (V^{\prime}/\langle b-1\rangle)^n+\langle b-1\rangle^n,
\end{split}
\end{equation*}
and
\begin{equation*}
(V/\langle c-1\rangle)^n\simeq (V^{\prime}/\langle b-1\rangle)^n,
\end{equation*}
we have
\begin{equation*}
f_V(n)=f_{V^{\prime}}(n) \text{ for all } n\geq 2.
\end{equation*}
Let $c$ be a real number such that 
\begin{equation*}
f_V(n),f_{V^{\prime}}(n)\leq c.
\end{equation*}
Let 
\begin{equation*}
R=\inf\{r\mid f_V(n)\leq cn^r\},R^{\prime}=\inf\{r\mid f_V^{\prime}(n)\leq cn^r\}.
\end{equation*}
Suppose $R<R^{\prime}$. Since ${\rm{GK}}\dim(k(M(L)))=r$ is finite, there exists a real number $c^{\prime}$ such that
\begin{equation*}
f_{V^{\prime}}(n)=f_V(n)\leq c^{\prime}n^r.
\end{equation*}
Thus we have
\begin{equation*}
f_{V^{\prime}}(1)\leq \max\{c,c^{\prime}\}, f_{V^{\prime}}(n)\leq \max\{c,c^{\prime}\}n^r.
\end{equation*}
Since $R<R^{\prime}$, this contradicts $R^{\prime}=\inf\{r\mid f_V^{\prime}(n)\leq cn^r\}$. Therefore, 
\begin{equation*}
{\rm{GK}}\dim(k(M(L)))={\rm{GK}}\dim(k(M(L^{\prime}))).
\end{equation*}
Thus ${\rm{GK}}\dim(k(M(L)))$ is invariant under the Reidemeister move ${\rm{I\hspace{-.1em}I}}$. Next let $L,L^{\prime}$ be the same links, except in the neighborhood of a point where they are as shown in Figure.
\begin{center}
\includegraphics[keepaspectratio,width=5cm,bb=0 0 250 270]{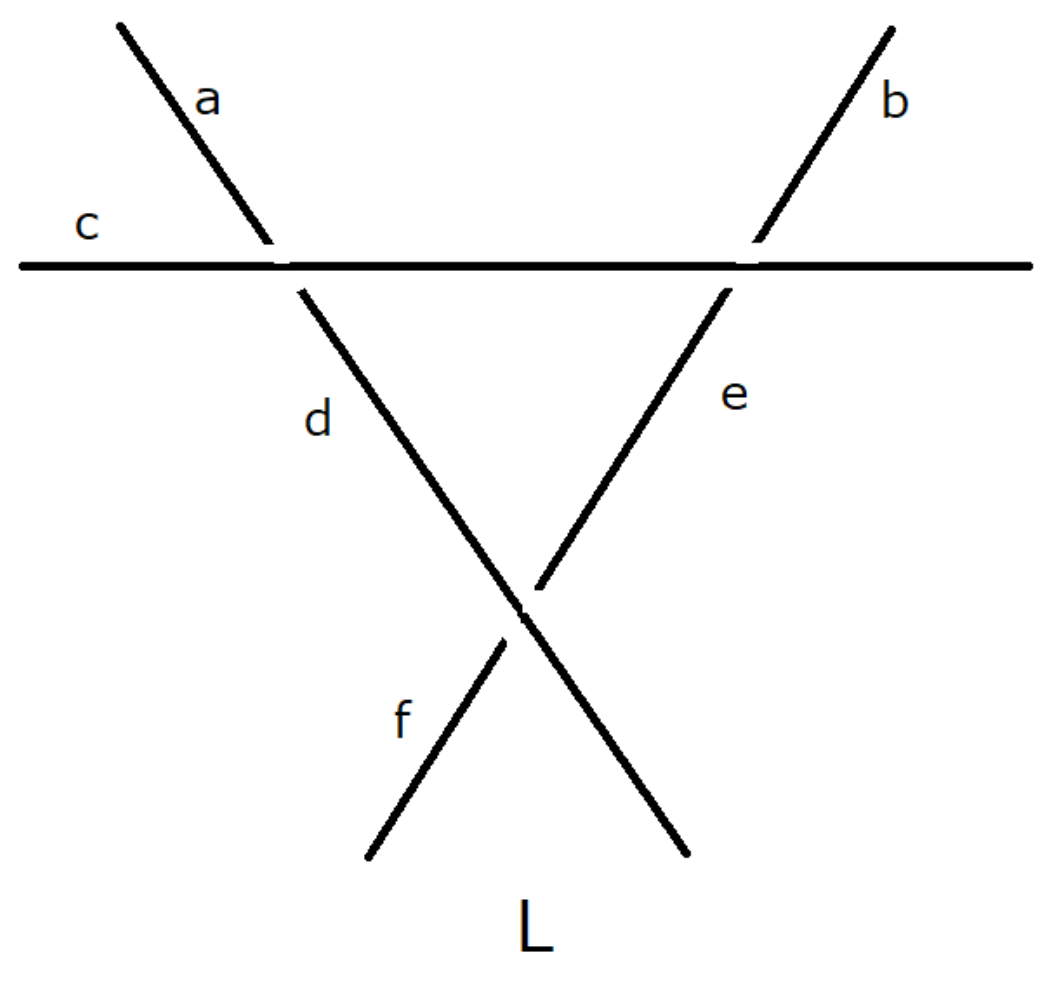} \,\,\,
\includegraphics[keepaspectratio,width=5cm,bb=0 0 250 230]{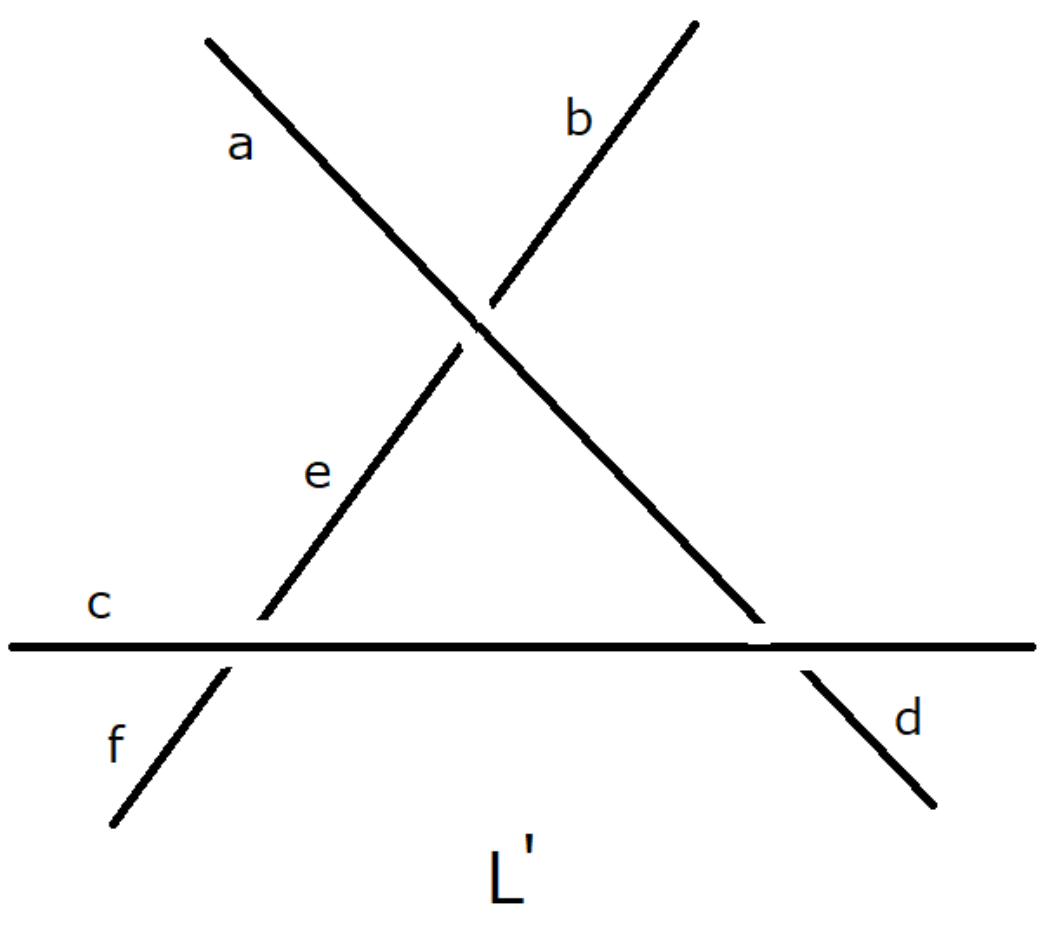}
\end{center}
We can assume ${\rm{GK}}\dim(k(M(L))),{\rm{GK}}\dim(k(M(L^{\prime})))<\infty$. Then 
\begin{equation*}
\begin{split}
&k(M(L))\simeq k\langle x_1,\dots,x_m,a,b,c,d,e,f\rangle/\langle I_{a,b,c,d,e,f}\cup J\rangle, \\
&k(M(L^{\prime}))\simeq k\langle x_1,\dots,x_m,a,b\rangle/\langle I^{\prime}_{a,b,c,d,e,f}\cup J^{\prime} \rangle,
\end{split}
\end{equation*}
where 
\begin{equation*}
\begin{split}
&J=\left\{
\begin{matrix}
ac-cd \\ ca-dc
\end{matrix}\right\}\cup
\left\{\begin{matrix}
bc-ce \\ cb-ec
\end{matrix}
\right\}\cup
\left\{\begin{matrix}
fd-de \\ df-ed
\end{matrix}
\right\}, \\
&J^{\prime}=\left\{
\begin{matrix}
ea-ab \\ ae-ba
\end{matrix}\right\}\cup
\left\{\begin{matrix}
fc-ce \\ cf-ec
\end{matrix}
\right\}\cup
\left\{\begin{matrix}
dc-cf \\ cd-fc
\end{matrix}
\right\},
\end{split}
\end{equation*}
and $x_1,\dots,x_m$ are labels of arcs except $a,b,c,d,e,f$, and $I_{a,b,c,d,e,f},I^{\prime}_{a,b,c,d,e,f}$ is a set of relations except $J,J^{\prime}$.
\begin{equation*}
\begin{split}
&V=\langle x_1,\dots,x_m,a,b,c,d,e,f\rangle/\langle I_{a,b,c,d,e,f}\cup J\rangle, \\
&V^{\prime}=\langle x_1,\dots,x_m,a,b,c,d,e,f\rangle/\langle I^{\prime}_{a,b,c,d,e,f}\cup J^{\prime}\rangle.
\end{split}
\end{equation*}
Then we have the following exact sequences for $n\geq 2$.
\begin{equation*}
\begin{split}
0\rightarrow \langle a-1\rangle^n \rightarrow V^n \rightarrow \left(V^n/\langle a-1\rangle \right)^n\rightarrow 0, \\
0\rightarrow \langle e-1\rangle^n \rightarrow V^n \rightarrow \left(V^n/\langle e-1\rangle \right)^n\rightarrow 0. 
\end{split}
\end{equation*}
Since 
\begin{equation*}
\begin{split}
&f_V(n)=\dim V^n=\dim (V/\langle a-1\rangle)^n+\langle a-1\rangle^n, \\
&f_{V^{\prime}}(n)=\dim {V^{\prime}}^n=\dim (V^{\prime}/\langle e-1\rangle)^n+\langle e-1\rangle^n,
\end{split}
\end{equation*}
and
\begin{equation*}
(V/\langle a-1\rangle)^n\simeq (V^{\prime}/\langle e-1\rangle)^n,
\end{equation*}
we have
\begin{equation*}
f_V(n)=f_{V^{\prime}}(n) \text{ for all } n\geq 2.
\end{equation*}
Then we can prove that ${\rm{GK}}\dim(k(M(L)))={\rm{GK}}\dim(k(M(L^{\prime})))$ by the similar way of the case of  the Reidemeister move ${\rm{I\hspace{-.1em}I}}$. Therefore ${\rm{GK}}\dim(k(M(L)))$ is invariant under the Reidemeister move {\rm{I\hspace{-.1em}I\hspace{-.1em}I}}.
\end{proof}
\subsection{Examples}
\begin{example}
Let $L_1$ be a Hopf link (Figure 2). Since 
\begin{equation*}
\begin{split}
k(M(L_1))&\simeq k\langle a,b\rangle/\langle ab-ba\rangle \\
&\simeq k[a,b],
\end{split}
\end{equation*}
${\rm{GK}}\dim(k(M(L_1)))=2$. On the other hand let $L_2$ be the following link.
 \begin{center}
\includegraphics[keepaspectratio,width=5cm,bb=0 0 250 160]{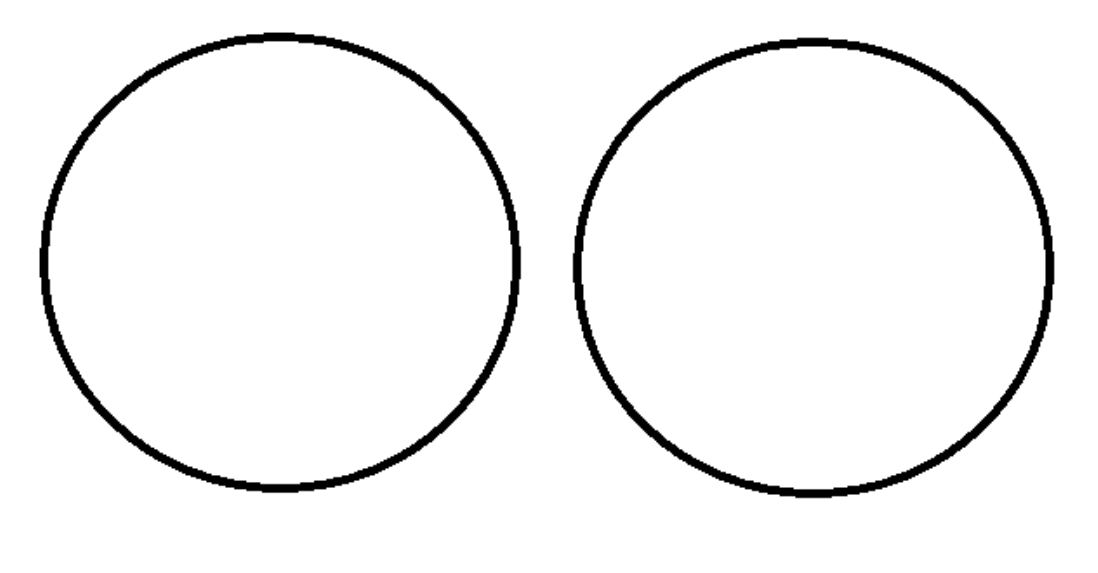}
\end{center}
Since 
\begin{equation}
k(M(L_2))\simeq k\langle a,b\rangle,
\end{equation}
${\rm{GK}}\dim(k(M(L_1)))=\infty$. Therefor we can conclude that $L_1\not\simeq L_2$.
\end{example}
\begin{example}
We shall consider torus knots $T(2,n)$ and double twist knots $\mathfrak{dtw}_m^l$. Let $V$ be a generating subspace of $k(M(T(2,n)))$. Then by Theorem \ref{torusthm}, 
\begin{equation*}
f_V(d)=\dim V^d=nd+1.
\end{equation*}
Thus ${\rm{GK}}\dim(k(M(T(2,n))))=1$. Let $V^{\prime}$ be a generating subspace of $k(M(\mathfrak{dtw}_m^l))$. The by theorem \ref{main},
\begin{equation*}
f_{V^{\prime}}(d)=
\begin{cases}
1 & d=0 \\
(lm+1)d+m+l-lm & d>0\,\,.
\end{cases}
\end{equation*}
Thus ${\rm{GK}}\dim(k(M(\mathfrak{dtw}_m^l)))=1$.
\end{example}

\end{document}